\documentclass[a4paper,12pt, oneside]{amsart}
\usepackage{mathtools}
\usepackage{color}
\usepackage{nicefrac}
\newtheorem{theorem}{Theorem}[section]
\newtheorem{lemma}[theorem]{Lemma}
\theoremstyle{definition}
\newtheorem{definition}[theorem]{Definition}

\newtheorem{corollary}[theorem]{Corollary}

\theoremstyle{remark}
\newtheorem{remark}[theorem]{Remark}
\numberwithin{equation}{section}

\newcommand{\ud}{\mathrm{d}}
\newcommand{\tv}{\mathrm{{\bf{d}}}_{\mathrm{TV}}}
\newcommand{\N}{\mathcal{N}}
\newcommand{\C}{\mathcal{C}}
\newcommand{\sg}{\mathrm{sign}}

\usepackage{anysize}
\marginsize{3cm}{2cm}{3cm}{2cm}
\begin{document}
\title{A switch convergence for a small perturbation of a linear recurrence equation}
\author{G. Barrera}
\address{University of Alberta, Department of Mathematical and Statistical Sciences. Central Academic Building. 116 Street and 85 Avenue. Postal Code: T6G--2G1. Edmonton, Alberta, Canada.}
\email{barrerav@ualberta.ca }
\thanks{The first author was supported by grant from University of Alberta, Department of Mathematical and Statistical Science and from Pacific Institute for the Mathematical Sciences, PIMS}
\author{S. Liu}
\address{University of Alberta, Department of Mathematical and Statistical Sciences. Central Academic Building. 116 Street and 85 Avenue. Postal Code: T6G--2G1. Edmonton, Alberta, Canada.}
\email{sliu5@ualberta.ca}
\thanks{The second author was supported by grant from University of Alberta, Department of Mathematical and Statistical Science.}
\keywords{Cut-off Phenomenon, Linear Recurrences, Gaussian Distribution, Total Variation Distance.}

\begin{abstract}
In this article we study a small random perturbation of a linear recurrence equation. If all the roots of its corresponding characteristic equation have modulus strictly less than one, the random linear recurrence goes exponentially fast to its limiting distribution  in the total variation distance as time increases. 
By assuming that all the roots of its corresponding characteristic equation have modulus strictly less than one and 
some suitable conditions, we prove that this convergence happens 
as a switch-type, i.e., there is a sharp transition in the convergence to its limiting distribution. 
This fact is known as a cut-off phenomenon in the context of stochastic processes.
\end{abstract}

\maketitle
\markboth{Cut-off phenomenon for random linear recurrences}{Cut-off phenomenon for random linear recurrences}
\section*{Introduction} 
Linear recurrence equations have been widely used in several areas of applied mathematics and computer science.
In applied science,
they can be used to model the future  of a process that depends linearly on a 
finite string, for instance:
in population dynamics to model population size and structure 
[\cite{BEN}, \cite{FD}, \cite{SW}]; 
in economics to model the interest rate, the amortization of a loan and price fluctuations 
[\cite{CEF}, \cite{MAR}, \cite{HK}]; in computer science for analysis of algorithms [\cite{CLRS}, \cite{JOL}]; in statistics for the autoregressive linear model [\cite{HA}, \cite{RD}].
In theoretical mathematics, for instance: 
in differential equations  to find the coefficients of series solutions [Chapters 4--5 in \cite{Cod}]; 
in the proof of Hilbert's tenth problem over $\mathbb{Z}$ \cite{Mati};
and in approximation theory to provide expansions of some second order operators \cite{Spig}. 
For a complete understanding of applications of the linear recurrence equations we recommend the Introduction of the monograph \cite{EPSW} and the references therein.

We consider a random dynamics that arises from a linear homogeneous recurrence equation with control term given by independent and identically distributed 
(i.i.d. for short) random variables  with Gaussian distribution. To be precise, 
{\it{given}} $p\in \mathbb{N}$, $\phi_1, \phi_2,\ldots, \phi_p\in\mathbb{R}$ with $\phi_p\not=0$,
we define the linear homogeneous recurrence of degree $p$ as follows:
\begin{equation}\tag{{\bf{L}}}\label{plr}
x_{t+p}=\phi_1x_{t+p-1}+\phi_2x_{t+p-2}+\cdots+\phi_p x_{t} \quad\textrm{  for any  } t\in \mathbb{N}_0,
\end{equation}
where $\mathbb{N}_0$ denotes the set of non-negative integers.
To single out a unique solution of \eqref{plr} one should
assign initial conditions $x_0,x_1,\ldots,x_{p-1}\in \mathbb{R}$.
Recurrence \eqref{plr} is called a recurrence with $p$-history since it only depends on a $p$-number of earlier values.

We consider a small perturbation of \eqref{plr} by adding Gaussian noise as follows: 
given $\epsilon>0$ fixed, consider the random dynamics
\begin{equation}\tag{{\bf{SL}}}\label{ar}
X^{(\epsilon)}_{t+p}=\phi_1X^{(\epsilon)}_{t+p-1}+\phi_2X^{(\epsilon)}_{t+p-2}+\cdots+\phi_p X^{(\epsilon)}_{t}+\epsilon \xi_{t+p}\quad\textrm{ for any } t\in \mathbb{N}_0,
\end{equation}
with initial conditions $X^{(\epsilon)}_0=x_0,X^{(\epsilon)}_1=x_1,\ldots,X^{(\epsilon)}_{p-1}=x_{p-1}$,
and $(\xi_t:t\geq p)$ is a sequence of i.i.d.  random variables with Gaussian distribution with zero mean and variance one.
Denote by $(\Omega,\mathcal{F},\mathbb{P})$ the probability space where the sequence $(\xi_t:t\geq p)$ is defined, 
then the random dynamics \eqref{ar} can be defined as a stochastic process in the probability space $(\Omega,\mathcal{F},\mathbb{P})$.

Notice that $\epsilon>0$ is parameter that controls the magnitude of the noise.
When $\epsilon=0$ the deterministic model \eqref{plr} recovers from the stochastic model \eqref{ar}. 
Since $(\xi_t:t\geq p)$ is a sequence of i.i.d.  random variables with Gaussian distribution,
the model \eqref{ar} could be understood as a regularization of \eqref{plr}. 

Up to our knowledge, 
this type of model was originally used in $1927$ by G. Yule  \cite{YUL} $(p=2)$, which models the presence of random disturbances of a harmonic oscillator  for investigating hidden periodicities and
its relation to the observations of sunspots.

In this article we obtain a {\it{nearly-complete characterization}} of the
convergence in the total variation distance between the {distribution of $X^{(\epsilon)}_t$} and its limiting distribution  as $t$ increases.
Under general conditions that we state in Section \ref{model}, when 
the intensity of the control 
$\epsilon$ is fixed, as the time goes by, the random linear recurrence goes to a limiting distribution in the total variation distance.
We show that this convergence is actually abrupt in the following sense: the total variation distance between the distribution of the random linear recurrence and its limiting distribution drops abruptly over a negligible time (time window) around a threshold time (cut-off time) from near one to near zero. It means that if we run the random linear recurrence before a time window around the cut-off time the process is not well mixed and after a time window around the cut-off time becomes well mixed. This fact is known as a cut-off phenomenon in the context of stochastic processes.

Suppose that we model  a system by a random process $(X^{(\epsilon)}_t:t\geq 0)$, where the parameter $\epsilon$ denotes the intensity of the noise and assume that $X^{(\epsilon)}_\infty$ is its equilibrium.
A natural question that arises is the following: 
{\it{with a {fixed} $\epsilon$ and an error $\eta>0$, how much time $\tau(\epsilon,\eta)$ do we need to run the model $(X^{(\epsilon)}_t:t\geq 0)$ in order to be close to its equilibrium $X^{(\epsilon)}_\infty$ by an error at most $\eta$ in a suitable distance?}}
The latter is known as a {\it{mixing time}} in the context of random processes. In general, it is hard to compute and/or estimate $\tau(\epsilon,\eta)$.
The cut-off phenomenon provides a strong answer in a small regime $\epsilon$. 
Roughly speaking,  as $\epsilon$ goes to zero, 
it means that in a deterministic time $\tau^{*}(\epsilon)$ the system is ``almost" in its equilibrium within any error $\eta$. We provide a precise definition in Section \ref{model}.

The cut-off phenomenon was extensively studied 
in the eighties to describe the phenomenon of abrupt convergence that appears in the models of cards' shuffling, Ehrenfests' urn and random transpositions, see for instance \cite{DIA}. In general, it is a challenging problem to prove that a specific model exhibits a cut-off phenomenon. It requires a complete understanding of the dynamics of the specific random process. For an introduction to this concept, we recommend Chapter $18$ of \cite{LP} for discrete Markov chains in a finite state, \cite{MART} for discrete Markov chains with infinite countable state space and [\cite{BJ}, \cite{BJ1}, \cite{BA}] for Stochastic Differential Equations in a continuous state space.

This article is organized as follows: {In} Section \ref{model}
we state the main result and its consequences.
In Section \ref{proof} we give the proof of  Theorem \ref{main} which is the main result of this article.
Also, we appoint conditions to verify the hypothesis of Theorem \ref{main}. In Section \ref{examples} we provide a complete understanding how to verify the conditions of Theorem \ref{main} for a discretization of 
the celebrated Brownian oscillator. Lastly, we provide Appendix \ref{tools} with some results about the distribution of the random linear recurrence and its limiting behavior, Appendix \ref{totalvariation} which summarizes some properties about the total variation distance between Gaussian distributions, and Appendix \ref{toolsappendix} which states some elementary limit behaviors.

\section{Main Theorem}\label{model} 
One of the most important problems in dynamical systems is the study of the limit behavior of its evolution for forward times.
To the linear recurrence \eqref{plr} we can associate a characteristic polynomial
\begin{equation}\label{ce}
f(\lambda)=  \lambda^{p}-\phi_1\lambda^{p-1}-\cdots-\phi_p \quad\textrm{  for any  } \lambda\in \mathbb{C}.
\end{equation}

From now to the end of this article, we assume
\begin{equation}
\tag{{\bf{H}}}
\label{H}
\textrm{all the roots of \eqref{ce} have modulus less than one.}
\end{equation}
From \eqref{H} we can  prove that 
for any string of initial values $x_0,\ldots,x_{p-1}\in \mathbb{R}$,  $x_{t}$ goes exponentially fast to zero as $t$ goes to infinity. For more details see Theorem 1 in \cite{GSL}.
In the stochastic model \eqref{ar}, \eqref{H} implies that the process $(X^{(\epsilon)}_t,t\in \mathbb{N}_0)$ is strongly ergodic, i.e., 
for any initial data $x_0,\ldots,x_{p-1}$, the random recurrence
 $X^{(\epsilon)}_t$ converges in the so-called total variation distance as $t$ goes to infinity to a random variable $X^{(\epsilon)}_\infty$. For further details see Lemma \ref{ergodico} in Appendix \ref{tools}.
 
Given $m\in \mathbb{R}$ and $\sigma^2\in (0,+\infty)$, denote by $\N(m,\sigma^2)$ the Gaussian distribution with mean $m$ and variance $\sigma^2$.
Later on, we see that for $t\geq p$ the random variable 
$X^{(\epsilon)}_t$ has distribution $\mathcal{N}(x_t,\epsilon^2 \sigma^2_t)$,
where $x_t$ is given by \eqref{plr} and $\sigma^2_t\in (0,+\infty)$.
Moreover,  the random variable $X^{(\epsilon)}_\infty$ has distribution 
$\N(0,\epsilon^2 \sigma^2_\infty)$ with $\sigma^2_\infty\in (0,+\infty)$.

Since the distribution of $X^{(\epsilon)}_t$ for $t\geq p$ and its limiting distribution
$X^{(\epsilon)}_\infty$ are absolutely continuous with respect to the Lebesgue measure on $\mathbb{R}$, a natural way to measure its discrepancy is by the total variation distance. Given two probability measures $\mathbb{P}_1$ and $\mathbb{P}_2$ on the  measure space $(\Omega,\mathcal{F})$, the total variation distance between the probabilities  $\mathbb{P}_1$ and $\mathbb{P}_2$ is given by
\begin{eqnarray*}\label{tv}
\tv(\mathbb{P}_1,\mathbb{P}_2)\coloneqq    \sup\limits_{F\in \mathcal{F}}|\mathbb{P}_1(F)-\mathbb{P}_2(F)|.
\end{eqnarray*}
When $X,Y$ are random variables defined in the probability space $(\Omega,\mathcal{F},\mathbb{P})$ we write 
$\tv(X,Y)$ instead of $\tv(\mathbb{P}(X\in\cdot), \mathbb{P}(Y\in\cdot))$, where $\mathbb{P}(X\in\cdot)$ and $\mathbb{P}(Y\in\cdot)$ denote the distribution of $X$ and $Y$ under $\mathbb{P}$, respectively.
Then we define
\begin{equation*}  \label{toma}
d^{(\epsilon)}(t)
:=\tv
\left(X^{(\epsilon)}_t,X^{(\epsilon)}_\infty\right)=
\tv\left(\mathcal{N}(x_t,\epsilon^2 \sigma^2_t),\mathcal{N}(0,\epsilon^2 \sigma^2_\infty)\right)\quad\textrm{ for any } t\geq p.
\end{equation*}
Notice that the above distance depends on the initial conditions $x_0,\ldots,x_{p-1}\in \mathbb{R}$. To do the notation more fluid, 
{\it{we avoid its dependence from our notation.}}
For a complete understanding of the total variation distance between two arbitrary probabilities with densities, 
we recommend Section $3.3$ in \cite{RDR} and Section $2.2$ in \cite{AD}. 
Nevertheless, for the shake of completeness, we provide an Appendix \ref{totalvariation} 
that contains the properties and bounds for the total variation distance between Gaussian distributions that we used to prove Theorem \ref{main}, which is the main theorem of this article.

 The goal is to study of the so-called {\it{cut-off phenomenon}} in the total variation distance when  $\epsilon$ goes to zero for the family of the stochastic processes 
\[\left(X^{(\epsilon)}:=\left(X^{(\epsilon)}_t:t\in \mathbb{N}_0\right): \epsilon>0\right)\]
for {\it{fixed}} initial conditions $x_0,\ldots,x_{p-1}$.

Roughly speaking, the argument of the proof consists in fairly intricate calculations of the distributions of $X^{(\epsilon)}_t$, $t\geq p$ and its limiting distribution $X^{(\epsilon)}_\infty$ whose distributions are Gaussian. Then the cut-off phenomenon is proved from a refined analysis of their means and variances, and  ``explicit calculations and bounds" for the total variation distance between Gaussian distributions.
This analysis also provides a delicate case in which the cut-off phenomenon does not occur.

Now, we introduce the formal definition of cut-off phenomenon.
Recall that for any $z\in \mathbb{R}$, $\lfloor z \rfloor$ denotes the greatest integer less than or equal to $z$. 
Consider the family of stochastic processes 
$(X^{(\epsilon)}:=(X^{(\epsilon)}_t:t\in \mathbb{N}_0): \epsilon>0)$.
According to \cite{BY}, the cut-off phenomenon can be expressed in three increasingly sharp levels as follows.
\begin{definition}
The family $(X^{(\epsilon)}:\epsilon>0)$ has 
\begin{itemize}
\item[i)] {\it{cut-off}} 
at $(t^{(\epsilon)}:\epsilon>0)$ with cut-off time  $t^{(\epsilon)}$
if
$t^{(\epsilon)}$ goes to infinity as $\epsilon$ goes to zero and
\[
\lim\limits_{\epsilon \rightarrow 0^+}d^{(\epsilon)}(\lfloor \delta t^{(\epsilon)}\rfloor)=
\begin{cases}
1 \quad&\textrm{if}~ 0<\delta<1,\\
0 &\textrm{if}~ \delta>1.
\end{cases}
\]
\item[ii)] 
{\it{window cut-off}} at $((t^{(\epsilon)},w^{(\epsilon)}): \epsilon>0)$
 with cut-off time $t^{(\epsilon)}$ 
and time cut-off $w^{(\epsilon)}$
if
$t^{(\epsilon)}$ goes to infinity as $\epsilon$ goes to zero, $w^{(\epsilon)}=o(t^{(\epsilon)})$ and
\[
\lim\limits_{b\rightarrow -\infty}\liminf\limits_{\epsilon \rightarrow 0^+}d^{(\epsilon)}(\lfloor t^{(\epsilon)}+bw^{(\epsilon)}\rfloor)=1
\quad \textrm{ and } \quad
\lim\limits_{b\rightarrow +\infty}\limsup\limits_{\epsilon \rightarrow 0^+}d^{(\epsilon)}(\lfloor t^{(\epsilon)}+bw^{(\epsilon)}\rfloor)=0. 
\]
\item[iii)] {\it{profile cut-off}} at 
$((t^{(\epsilon)},w^{(\epsilon)}):\epsilon>0)$ 
with cut-off time $t^{(\epsilon)}$, 
time cut-off $w^{(\epsilon)}$ and
profile function $G:\mathbb{R}\rightarrow [0,1]$
if 
$t^{(\epsilon)}$ goes to infinity as $\epsilon$ goes to zero, $w^{(\epsilon)}=o(t^{(\epsilon)})$,
\[
\lim\limits_{\epsilon \rightarrow 0^+}d^{(\epsilon)}(\lfloor t^{(\epsilon)}+bw^{(\epsilon)}\rfloor)=:G(b) \quad \textrm{ exists for any } b\in \mathbb{R}
\]
together with
 $\lim\limits_{b\rightarrow -\infty}G(b)=1$ and 
$\lim\limits_{b\rightarrow +\infty}G(b)=0$.
\end{itemize}
\end{definition}

Bearing all this in mind, we can analyze how this convergence happens which is exactly the statement  of the following theorem.
\begin{theorem}[Main theorem]\label{main}
Assume that \eqref{H} holds.
For a given initial data $x=(x_0,\ldots,x_{p-1})\in \mathbb{R}^p\setminus\{0_p\}$ 
assume that there exist
$r=r(x)\in (0,1)$, $l=l(x)\in \{1,\ldots,p\}$ and $v_t=v(t,x)\in \mathbb{R}$ such that 
\begin{enumerate}
\item[i)] 
\begin{equation*}\label{most}
\lim\limits_{t\rightarrow +\infty}\left|\frac{x_t}{t^{l-1}r^t}-v_t\right|=0,
\end{equation*}
\item[ii)] $\sup\limits_{t\rightarrow+\infty}|v_t|<+\infty$,
\item[iii)] $\liminf\limits_{t\rightarrow+\infty}|v_t|>0$.
\end{enumerate}
Then the family of random linear recurrences $(X^{(\epsilon)}:=(X^{(\epsilon)}(t):t\in \mathbb{N}_0): \epsilon>0)$ has {window} cut-off as $\epsilon$ goes to zero with
cut-off time 
\[
t^{(\epsilon)}=\frac{\ln(\nicefrac{1}{\epsilon})}{\ln(\nicefrac{1}{r})}+
(l-1)\frac{\ln\left(\frac{\ln(\nicefrac{1}{\epsilon})}{\ln(\nicefrac{1}{r})}\right)}{\ln(\nicefrac{1}{r})}
\]
and time window 
\[
w^{(\epsilon)}=C+o_{\epsilon}(1),
\]
where $C$ is any positive constant and $\lim\limits_{\epsilon\rightarrow 0^+}o_{\epsilon}(1)=0$. 
In other words,
\[
\lim\limits_{b\rightarrow -\infty}\liminf\limits_{\epsilon \rightarrow 0^+}d^{(\epsilon)}(\lfloor t^{(\epsilon)}+bw^{(\epsilon)}\rfloor)=1
\textrm{ and }
\lim\limits_{b\rightarrow +\infty}\limsup\limits_{\epsilon \rightarrow 0^+}d^{(\epsilon)}(\lfloor t^{(\epsilon)}+bw^{(\epsilon)}\rfloor)=0,
\]
where $d^{(\epsilon)}(t)=\tv
\left(X^{(\epsilon)}_t,X^{(\epsilon)}_\infty\right)$ for any $t\geq p$.
\end{theorem}

\begin{remark}
Notice that 
$\sup\limits_{t\rightarrow+\infty}|v_t|<+\infty$ and 
$\limsup\limits_{t\rightarrow+\infty}|v_t|<+\infty$ are actually equivalent. However,  
$\liminf\limits_{t\rightarrow+\infty}|v_t|>0$ does not always imply
$\inf\limits_{t\geq 0}|v_t|>0$.
\end{remark}

\begin{remark}
Roughly speaking, the number $r$
corresponds to the absolute value of some roots of \eqref{ce} and 
$l$ is related to their multiplicities.
\end{remark}

\begin{remark}
Under the conditions of Theorem \ref{main}, the total variation distance between the distribution of $X^{(\epsilon)}_t$ and its limiting distribution $X^{(\epsilon)}_\infty$ drives abruptly from one to zero  in a time window $w^{(\epsilon)}$ of constant order around the cut-off time $t^{(\epsilon)}$ of logarithmic order.
\end{remark}

We introduce the definition of maximal set.
We say that a set $\mathcal{A}\subset \mathbb{R}^{p}$ is a maximal set that satisfies the property {\bf{P}} if and only if any set $\mathcal{B}\subset \mathbb{R}^d$ that satisfies the property {\bf{P}} is a subset of $\mathcal{A}$.

In the case when all the roots of \eqref{ce} are real numbers we see in Lemma \ref{lya} that there exists a maximal set $\C\subset \mathbb{R}^p$ such that any initial datum $x:=(x_0,\ldots,x_{p-1}) \in \C$ fulfills Condition i), Condition ii) and Condition iii) of Theorem \ref{main}. 
Moreover,  $\C$ has full measure with respect to the Lebesgue measure on $\mathbb{R}^p$.  
If we only assume \eqref{H} and {\it{no further assumptions}}, we see in Corollary \ref{comp} that Condition iii) of Theorem \ref{main} may not hold. 

\section{Proof}\label{proof}
Since the random recurrence \eqref{ar} is linear on the inputs which are independent Gaussian random variables, the time distribution of the random dynamics for $t\geq p$ is also Gaussian.
Observe that for any $t\geq p$, $X^{(\epsilon)}_t$ has Gaussian distribution with mean $x_t$ and variance
$\sigma^2(t,\epsilon, x_0,\ldots,x_{p-1})\in (0,+\infty)$.
Later on, in Lemma \ref{distri} in Appendix \ref{tools}, under assumption \eqref{H}, we  see that
$\sigma^2(t,\epsilon, x_0,\ldots,x_{p-1})=\epsilon^2\sigma^2_t$, where
$\sigma^2_t\in [1,+\infty)$ and it does not depend on the initial data $x_0,x_1,\ldots,x_{p-1}$.

The following lemma asserts that the random dynamics \eqref{ar} is strongly ergodic when \eqref{H} holds.

\begin{lemma}\label{lemma2}
Assume that \eqref{H} holds.
As $t$ goes to infinity, $X^{(\epsilon)}_t$ converges in the total variation distance to a random variable
$X^{(\epsilon)}_\infty$ that has Gaussian distribution with zero mean and variance $\epsilon^2\sigma^2_\infty\in [\epsilon^2,+\infty)$.
\end{lemma}
For the sake of brevity, the proof of the last lemma is given in Lemma \ref{ergodico} in Appendix \ref{tools}.
Recall that
\begin{equation*} 
d^{(\epsilon)}(t)=
\tv\left(\mathcal{N}(x_t,\epsilon^2 \sigma^2_t),\mathcal{N}(0,\epsilon^2 \sigma^2_\infty)\right) \quad \textrm{ for any } t\geq p.
\end{equation*}
In order to analyze the cut-off phenomenon for the distance $d^{(\epsilon)}(t)$, 
for the convenience of computations
we turn to study another distance as the following lemma states.

\begin{lemma}\label{opo}
For any $t\geq p$ we have
\begin{equation*}\label{buenasi}
\left| d^{(\epsilon)}(t)-D^{(\epsilon)}(t)\right|\leq R(t)
\end{equation*}
where
\[D^{(\epsilon)}(t)=\tv\left(\mathcal{N}\left(\frac{x_t}{\epsilon\sigma_\infty},1\right), \mathcal{N}(0,1)\right)\]
 and 
 \[
 R(t)=\tv(\mathcal{N}(0,\sigma^2_t),\mathcal{N}(0,\sigma^2_\infty)).
 \] 
\end{lemma}

\begin{proof} 
Notice that the terms $d^{(\epsilon)}(t)$ and 
$D^{(\epsilon)}(t)$ depend on the parameter $\epsilon$ and the initial data $x_0,x_1,\ldots,x_{p-1}$.
Nevertheless, the term $R(t)$ does not depend on $\epsilon$ and on the initial data $x_0,x_1,\ldots,x_{p-1}$. 
Let $t\geq p$.
By the triangle inequality we obtain
\begin{align*}
d^{(\epsilon)}(t)\leq \tv\left(\N(x_t,\epsilon^2\sigma^2_t),\N(x_t,\epsilon^2 \sigma^2_\infty)\right)+
 \tv\left(\N(x_t,\epsilon^2\sigma^2_\infty),\N(0,\epsilon^2 \sigma^2_\infty)\right).
\end{align*}
By item i) and item ii) of  Lemma \ref{pend} we have
\begin{align*}
d^{(\epsilon)}(t)\leq R(t)+D^{(\epsilon)}(t).
\end{align*}
On the other hand, by item ii) of Lemma \ref{pend} we notice
\[D^{(\epsilon)}(t)=
\tv\left(\N(x_t,\epsilon^2\sigma^2_\infty),\N(0,\epsilon^2 \sigma^2_\infty)\right).\]
By the triangle inequality we obtain
\begin{align*}
D^{(\epsilon)}(t)\leq \tv\left(\N(x_t,\epsilon^2\sigma^2_\infty),\N(x_t,\epsilon^2 \sigma^2_t)\right)+
 \tv\left(\N(x_t,\epsilon^2\sigma^2_t),\N(0,\epsilon^2 \sigma^2_\infty)\right).
\end{align*}
Again, by item i) and item ii) of  Lemma \ref{pend} we have
\begin{align*}
D^{(\epsilon)}(t)\leq R(t)+d^{(\epsilon)}(t).
\end{align*}
Gluing all pieces together we deduce 
\begin{equation*}
\left| d^{(\epsilon)}(t)-D^{(\epsilon)}(t)\right|\leq R(t)\quad \textrm{ for any } t\geq p.
\end{equation*}
\end{proof}

Now, we have all the tools to prove Theorem \ref{main}.

\begin{proof}[Proof of Theorem \ref{main}]
By Lemma \ref{lemma2} and 
Lemma \ref{pend1} we have $\lim\limits_{t\rightarrow +\infty}R(t)=0$.
In order to analyze  $D^{(\epsilon)}(t)$ we observe that
\begin{equation}\label{ton1}
\frac{x_t}{\epsilon \sigma_\infty}=
\frac{t^{l-1}r^{t}}{\epsilon\sigma_\infty}
\left(\frac{x_t}{t^{l-1}r^t}-v_t\right)+\frac{t^{l-1}r^{t}}{\epsilon\sigma_\infty}v_t,
\end{equation}
where $l\in \{1,\ldots,p\}$, $r\in(0,1)$, and $v_t$ are given by Condition i).
By Lemma \ref{C3} in {Appendix \ref{toolsappendix}} we have
\[
\lim\limits_{\epsilon\rightarrow0^+}\frac{(t^{(\epsilon)})^{l-1} r^{t^{(\epsilon)}}}{\epsilon} = 1.
\]
For any $t\geq 0$, define 
$p_t=\frac{t^{l-1}r^{t}}{\epsilon\sigma_\infty}
\left(\frac{x_t}{t^{l-1}r^t}-v_t\right)$
and 
$q_t=\frac{t^{l-1}r^{t}}{\epsilon\sigma_\infty}v_t$.
Then for any $b\in \mathbb{R}$ we have
\begin{align*}
|p_{\lfloor t^{(\epsilon)}+bw^{(\epsilon)}\rfloor}|\leq &
\left(\frac{t^{(\epsilon)}+bw^{(\epsilon)}}{t^{(\epsilon)}}\right)^{l-1}
\frac{(t^{(\epsilon)})^{l-1}r^{t^{(\epsilon)}+bw^{(\epsilon)}-1}}{\epsilon\sigma_\infty}\times
\\
&
\left|
\frac{x_{\lfloor t^{(\epsilon)}+bw^{(\epsilon)}\rfloor}}{({\lfloor t^{(\epsilon)}+bw^{(\epsilon)}\rfloor})^{l-1}r^{{\lfloor t^{(\epsilon)}+bw^{(\epsilon)}\rfloor}}}-v_{\lfloor t^{(\epsilon)}+bw^{(\epsilon)}\rfloor}
\right|.
\end{align*}
By Condition i) we have
\begin{equation}\label{ton2}
\lim\limits_{\epsilon\rightarrow 0^+}
p_{\lfloor t^{(\epsilon)}+bw^{(\epsilon)}\rfloor}=0 \quad\textrm{ for any } b\in \mathbb{R}.
\end{equation}
Now, we analyze an upper bound for $|q_{\lfloor t^{(\epsilon)}+bw^{(\epsilon)}\rfloor}|$.
Notice that
\begin{equation*}
|q_{\lfloor t^{(\epsilon)}+bw^{(\epsilon)}\rfloor}|\leq
\left(\frac{t^{(\epsilon)}+bw^{(\epsilon)}}{t^{(\epsilon)}}\right)^{l-1}
\frac{(t^{(\epsilon)})^{l-1}r^{t^{(\epsilon)}+bw^{(\epsilon)}-1}}{\epsilon\sigma_\infty}M,
\end{equation*}
where $M=\sup\limits_{t\geq 0}|v_t|$. By Condition ii) we know $M<+\infty$. 
Then 
\begin{equation}\label{ton3}
\limsup\limits_{\epsilon \rightarrow 0^+}|q_{\lfloor t^{(\epsilon)}+bw^{(\epsilon)}\rfloor}|\leq \frac{Mr^{bC-1}}{\sigma_\infty} \quad\textrm{ for any } b\in \mathbb{R}.
\end{equation}
From equality \eqref{ton1},  relation \eqref{ton2}, inequality \eqref{ton3} and item ii) of Lemma \ref{C2} we get
\[
\limsup\limits_{\epsilon \rightarrow 0^+}
\frac{|x_{\lfloor t^{(\epsilon)}+bw^{(\epsilon)}\rfloor}|}{\epsilon\sigma_\infty}\leq \frac{Mr^{bC-1}}{\sigma_\infty} \quad\textrm{ for any } b\in \mathbb{R}.
\]
Using item i) of Lemma \ref{pend4} we have 
\begin{align*}
\limsup\limits_{\epsilon \rightarrow 0^+}& ~
\tv\left(\N\left(\frac{|x_{\lfloor t^{(\epsilon)}+bw^{(\epsilon)}\rfloor}|}{\epsilon\sigma_\infty},1\right),\N(0,1)\right)
\leq \\
&~  \tv\left(\N\left(\frac{Mr^{bC-1}}{\sigma_\infty},1\right),\N(0,1)\right)
\end{align*}
for any $b\in \mathbb{R}$.
Since $r\in (0,1)$, then by 
Lemma \ref{pend1} we have
\begin{equation}\label{b1}
\lim\limits_{b\rightarrow +\infty}\limsup\limits_{\epsilon \rightarrow 0^+}
\tv\left(\N\left(\frac{|x_{\lfloor t^{(\epsilon)}+bw^{(\epsilon)}\rfloor}|}{\epsilon\sigma_\infty},1\right),\N(0,1)\right)=0.
\end{equation}
In order to analyze a lower bound for $|q_{\lfloor t^{(\epsilon)}+bw^{(\epsilon)}\rfloor}|$, note
\begin{align*}
|q_{\lfloor t^{(\epsilon)}+bw^{(\epsilon)}\rfloor}|&\geq 
\left(\frac{t^{(\epsilon)}+bw^{(\epsilon)}-1}{t^{(\epsilon)}}\right)^{l-1}
\frac{(t^{(\epsilon)})^{l-1}r^{t^{(\epsilon)}+bw^{(\epsilon)}}}{\epsilon\sigma_\infty}
|v_{\lfloor t^{(\epsilon)}+bw^{(\epsilon)}\rfloor}|\\
\end{align*}
for any $b\in \mathbb{R}$.
By Condition iii) and item iii) of Lemma \ref{C2} we have
\begin{equation}\label{ton4}
\liminf\limits_{\epsilon \rightarrow 0^+}
|q_{\lfloor t^{(\epsilon)}+bw^{(\epsilon)}\rfloor}|\geq 
\frac{r^{bC}}{\sigma_\infty}
\liminf\limits_{\epsilon \rightarrow 0^+}|v_{\lfloor t^{(\epsilon)}+bw^{(\epsilon)}\rfloor}|\geq \frac{mr^{bC}}{\sigma_\infty},
\end{equation}
where
$m=\liminf\limits_{t\rightarrow +\infty}|v_t|\in (0,+\infty)$.
From equality \eqref{ton1}, relation \eqref{ton2}, inequality \eqref{ton4} and item ii) of Lemma \ref{C2} we get
\begin{equation*}
\liminf\limits_{\epsilon \rightarrow 0^+}
\frac{|x_{\lfloor t^{(\epsilon)}+bw^{(\epsilon)}\rfloor}|}{\epsilon\sigma_\infty}\geq 
\frac{mr^{bC}}{\sigma_\infty} \quad\textrm{ for any } b\in \mathbb{R}.
\end{equation*}
From item ii) of Lemma \ref{pend4} we have
\begin{align*}
\liminf\limits_{\epsilon \rightarrow 0^+}&~
\tv\left(\N\left(\frac{|x_{\lfloor t^{(\epsilon)}+bw^{(\epsilon)}\rfloor}|}{\epsilon\sigma_\infty},1\right),\N(0,1)\right)
\geq  \\
& ~\tv\left(\N\left(\frac{r^{bC}}{\sigma_\infty}
m,1\right),\N(0,1)\right)
\end{align*}
for any $b\in \mathbb{R}$.
Since $r\in (0,1)$, then by item iii) Lemma \ref{pend2} we have
\begin{equation}\label{b2}
\lim\limits_{b\rightarrow -\infty}\liminf\limits_{\epsilon \rightarrow 0^+}
\tv\left(\N\left(\frac{|x_{\lfloor t^{(\epsilon)}+bw^{(\epsilon)}\rfloor}|}{\epsilon\sigma_\infty},1\right),\N(0,1)\right)=1.
\end{equation}
From  \eqref{b1} and \eqref{b2} we have
\[
\lim\limits_{b\rightarrow +\infty}\limsup\limits_{\epsilon\rightarrow 0^+}D^{(\epsilon)}(\lfloor  t^{(\epsilon)}+bw^{(\epsilon)} \rfloor)=0
\textrm{ and }
\lim\limits_{b\rightarrow -\infty}\liminf\limits_{\epsilon\rightarrow 0^+}D^{(\epsilon)}(\lfloor t^{(\epsilon)}+bw^{(\epsilon)}\rfloor)=1.
\]
Recall that
$
\lim\limits_{t\rightarrow +\infty}R(t)=0
$.
By Lemma \ref{opo} and  item i) of Lemma \ref{C2} we obtain 
\[
\limsup\limits_{\epsilon\rightarrow 0^+}d^{(\epsilon)}(\lfloor t^{(\epsilon)}+bw^{(\epsilon)}\rfloor)\leq \limsup\limits_{\epsilon\rightarrow 0^+}D^{(\epsilon)}(\lfloor t^{(\epsilon)}+bw^{(\epsilon)}\rfloor).
\]
Now, sending $b \to +\infty$ we get
\[
\lim\limits_{b\rightarrow +\infty}\limsup\limits_{\epsilon\rightarrow 0^+}d^{(\epsilon)}(\lfloor t^{(\epsilon)}+bw^{(\epsilon)}\rfloor)=0.
\]
Similarly, by Lemma \ref{opo} and  item ii) of Lemma \ref{C2} we obtain 
\[
\liminf\limits_{\epsilon\rightarrow 0^+}D^{(\epsilon)}(\lfloor t^{(\epsilon)}+bw^{(\epsilon)}\rfloor)\leq \liminf\limits_{\epsilon\rightarrow 0^+}d^{(\epsilon)}(\lfloor t^{(\epsilon)}+bw^{(\epsilon)}\rfloor).
\]
Now, sending $b \to -\infty$ we get
\[
\lim\limits_{b\rightarrow -\infty}\liminf\limits_{\epsilon\rightarrow 0^+}d^{(\epsilon)}(\lfloor t^{(\epsilon)}+bw^{(\epsilon)}\rfloor)=1.
\]
\end{proof}

\subsection{Fulfilling the conditions of Theorem \ref{main}}
Now, we provide a precise estimate of the rate of the convergence to zero of \eqref{plr}. Let us recall some well-known facts about $p$-linear recurrences. 
By the celebrated Fundamental Theorem of Algebra we have at most $p$ roots in the complex numbers for  \eqref{ce}. Denote by $\lambda_1,\ldots,\lambda_q\in \mathbb{C}$ the different roots of \eqref{ce} with multiplicity $m_1,\ldots,m_q$ respectively, where $1\leq q\leq p$.
Then 
\begin{equation}\label{eq1}
x_t=\sum\limits_{j_1=1}^{m_1}c_{1,j_1}t^{j_1-1}\lambda_1^t+\sum\limits_{j_2=1}^{m_2}c_{2,j_2}t^{j_2-1}\lambda_2^t+
\ldots+\sum\limits_{j_q=1}^{m_q}c_{q,j_q}t^{j_q-1}\lambda_q^t
\end{equation}
for any $t\in \mathbb{N}_0$, where the coefficients $c_{1,1},\ldots,c_{1,m_1},\ldots,c_{q,1},\ldots,c_{q,m_q}$ are uniquely obtained from the initial data $x_0,\ldots,x_{p-1}$. For more details see Theorem $1$  in \cite{GSL}. Moreover, for any initial data $(x_0,\ldots,x_{p-1})\in\mathbb{R}^{p}\setminus \{0_p\}$ we have 
\[(c_{1,1},\ldots,c_{1,m_1},\ldots,c_{q,1},\ldots,c_{q,m_q})\in \mathbb{C}^{p}\setminus \{0_p\}.\] 

Notice that the right-hand side of \eqref{eq1} may have complex numbers. When all the roots of  \eqref{ce} are real numbers we can establish the precise exponential behavior of $x_t$  as $t$ goes by.

\begin{lemma}[Real roots]\label{lya}
{Assume that all the roots of \eqref{ce} are real numbers}. Then
there exists a non-empty maximal set $\C \subset \mathbb{R}^p$ such that for any $x=(x_0,\ldots,x_{p-1})\in \C$ there exist
$r\coloneqq     r(x)>0$, $l\coloneqq     l(x)\in \{1,\ldots,p\}$ and $v_t\coloneqq     v(t,x)\in \mathbb{R}$ satisfying 
\begin{equation*}\label{lyapunov}
\lim\limits_{t\rightarrow +\infty}\left|\frac{x_t}{t^{l-1}r^t}-v_t\right|=0.
\end{equation*}
Moreover, we have  $\sup\limits_{t\rightarrow+\infty}|v_t|<+\infty$ and $\liminf\limits_{t\rightarrow+\infty}|v_t|>0$. 
\end{lemma}

\begin{proof}
Recall that the constants $c_{1,1},\ldots,c_{1,m_1},\ldots,c_{q,1},\ldots,c_{q,m_q}$ in  representation \eqref{eq1} depend on the initial data $x_0,x_1,\ldots,x_{p-1}$.
In order to avoid technicalities, without loss of generality we can assume that for each $1\leq j\leq q$ there exists at least one $1\leq k\leq m_j$ such that $c_{j,j_k}\not=0$. If the last assumption is not true for some $1\leq j\leq q$, then the root $\lambda_j$ does not appear in representation \eqref{eq1} for an specific initial data $x_0,x_1,\ldots,x_{p-1}$, then we can remove from representation \eqref{eq1} and apply the method described below.  

Denote by  $r=\max\limits_{1\leq j\leq q}{|\lambda_j|}>0$. Since all the roots of \eqref{ce} are real numbers then after multiplicity at most two roots of \eqref{ce} have the same absolute value. The function $\sg(\cdot)$ is 
defined over the domain $\mathbb{R}\setminus\{0\}$ by
$\sg(x)=\nicefrac{x}{|x|}$.
Only one of the following cases can occur.
\begin{itemize}
\item[i)] There exists a unique $1\leq j\leq q$ such that $|\lambda_{j}|=r$.
Let 
\[l=\max\{1\leq s\leq m_j:c_{j,s}\not=0\}.\] 
Then 
\[\lim\limits_{t\rightarrow +\infty}\left|\frac{x_t}{t^{l-1}r^t}-c_{j,l}(\sg (\lambda_j))^t\right|=0.\]
In this case $\C=\mathbb{R}^p\setminus \{0_p\}$.
\item[ii)] There exist $1\leq j<k\leq q$ such that $|\lambda_{j}|=|\lambda_k|=r$. Without loss of generality, we can assume $0<\lambda_k=-\lambda_j$.
Let 
\[l_j=\max\{1\leq s\leq m_j:c_{j,s}\not=0\}
\] 
and 
\[l_k=\max\{1\leq s\leq m_k:c_{k,s}\not=0\}.
\] 
If $l_j<l_k$ or $l_k<l_j$ then 
by taking $l=\max\{l_j,l_k\}$ we have
\[
\lim\limits_{t\rightarrow +\infty}\left|\frac{x_t}{t^{l-1}r^t}-c_{\star,l}(\sg (\lambda_\star))^t\right|=0,
\] where $\star=j$ if $l_j=l$ and
 $\star=k$ if $l_k=l$. In this case $\C=\mathbb{R}^p\setminus \{0_p\}$.
If $l_j=l_k$ then by taking $l=l_j$, $v_t=(-1)^{t}c_{j,l}+c_{k,l}$ we have
\begin{equation*}
\lim\limits_{t\rightarrow +\infty}\left|\frac{x_t}{t^{l-1}r^t}-v_t\right|=0.
\end{equation*}
Notice that $\sup\limits_{t\geq 0}|v_t|< +\infty$.
By taking 
\[\C=\{(x_0,\ldots,x_{p-1})\in \mathbb{R}^p: -c_{j,l}+c_{k,l}\not=0 \textrm{ and } c_{j,l}+c_{k,l}\not=0 \}\] we have $\liminf\limits_{t\rightarrow+\infty}|v_t|>0$.
\end{itemize}
\end{proof}

\begin{remark}
From the proof of Lemma \ref{lya}, we can state precisely $\C$. Moreover, $\C$ has full measure with respect to the Lebesgue measure on $\mathbb{R}^p$.
\end{remark}

Rather than the real roots case, the following lemma provides a fine estimate about the behavior of \eqref{plr} as $t$ goes by in general setting.

\begin{lemma}[General case]\label{ulya33}
For any $x=(x_0,\ldots,x_{p-1})\in \mathbb{R}^p\setminus \{0_{p}\}$
 there exist
$r:=r(x)>0$, $l\coloneqq    l(x)\in \{1,\ldots,p\}$ and $v_t\coloneqq    v(t,x)\in \mathbb{R}$ such that 
\begin{equation*}
\lim\limits_{t\rightarrow +\infty}\left|\frac{x_t}{t^{l-1}r^t}-v_t\right|=0,
\end{equation*}
where
\[
v_t=\sum\limits_{j=1}^{m}
\left(\alpha_j\cos(2\pi\theta_j t)+\beta_j\sin(2\pi\theta_j t)\right)
\]
with $(\alpha_j, \beta_j) \coloneqq   ( \alpha_j(x), \beta_j(x)) \in \mathbb{R}^2\setminus\{(0,0)\}$, $m\coloneqq     m(x)\in \{1,\ldots,p\}$, and $\theta_j\coloneqq \theta(x)\in [0,1)$ for any $j \in \{ 1, \ldots , m\}$. Moreover, $\sup\limits_{t \geq 0}|v_t|<+\infty$.
\end{lemma}
\begin{proof}
From \eqref{eq1} we have
\[
x_t=\sum\limits_{j_1=1}^{m_1}c_{1,j_1}t^{j_1-1}\lambda_1^t+\sum\limits_{j_2=1}^{m_2}c_{2,j_2}t^{j_2-1}\lambda_2^t+
\ldots+\sum\limits_{j_q=1}^{m_q}c_{q,j_q}t^{j_q-1}\lambda_q^t \quad\textrm{ for any } t\in \mathbb{N}_0.
\]
Without loss of generality we assume  for any $k\in\{1, \ldots, q\}$ there exists $j \in \{1,\ldots, m_k\}$ such that $c_{k,j}\neq 0$. Let $l_k\coloneqq     \max\{1\leq j\leq m_k:c_{k,j}\neq 0\}$.  Then $x_t$ can be rewritten as
\[
x_t=\sum\limits_{j_1=1}^{l_1}c_{1,j_1}t^{j_1-1}\lambda_1^t+\sum\limits_{j_2=1}^{l_2}c_{2,j_2}t^{j_2-1}\lambda_2^t+
\ldots+\sum\limits_{j_q=1}^{l_q}c_{q,j_q}t^{j_q-1}\lambda_q^t,
\]
where $c_{k,l_k}\neq 0$ for each $k$. For each $k$ let $r_k \coloneqq     \| \lambda_k \|$ be its complex modulus. Without loss of generality we assume:
\begin{itemize}
\item[i)] $r_1\leq \cdots \leq r_q$,
\item[ii)] there exists an integer $\tilde{h}$ such that $r_{\tilde{h}}=\cdots=r_q$,
\item[iii)] $l_{\tilde{h}} \leq \cdots \leq l_q$,
\item[iv)] there exists an integer $h\geq \tilde{h}$ such that $l_h=\cdots=l_q$.
\end{itemize}
Let $r\coloneqq     r_q$ and $l\coloneqq     l_q$. By taking $v_t=r^{-t}(c_{h,l}\lambda_h^t+\dots+c_{q,l}\lambda_q^t)$ we have
\[
\lim\limits_{t\rightarrow +\infty}\left|\frac{x_t}{t^{l-1}r^t}-v_t\right|=0,
\]
where $\lambda_h,\dots,\lambda_q$ have the same modulus $r$, but they have different arguments $\theta_j\in [0,1)$. Then
\[
v_t=\sum\limits_{j=h}^{q}
\left(\alpha_j\cos(2\pi\theta_j t)+\beta_j\sin(2\pi\theta_j t)\right).
\]
Since $c_{k,l_k} \neq 0$ for each $h\leq k\leq q$, then $\alpha_j$ and $\beta_j$ are not both zero for any $h\leq j\leq q$. After relabeling we have the desired result.
\end{proof}

\begin{remark}
Under no further conditions on Lemma \ref{ulya33}, we cannot guarantee that $\liminf\limits_{t\rightarrow +\infty}|v_t|>0$.  
For instance, the following corollary provides sufficient conditions for which $\liminf\limits_{t\rightarrow +\infty}|v_t|=0$.
\end{remark}

Following \cite{MK}, we define that the numbers
$\vartheta_1,\ldots,\vartheta_m $ are rationally independent if 
the linear combination $k_1\vartheta_1+\ldots+k_m\vartheta_m \notin \mathbb{Z}$ for any $(k_1, \ldots, k_m) \in \mathbb{Z}^m \setminus \{ 0_m\}$.
\begin{corollary}\label{comp}
Assume that $\theta_1,\ldots,\theta_m $ are rationally independent
 then $\liminf\limits_{t \rightarrow +\infty} |v_t |=0$.
\end{corollary}

\begin{proof}
For any $j\in\{1,\ldots,m\}$ notice that $d_j:=\sqrt{\alpha_j^2+\beta_j^2}>0$, and let $\cos(\gamma_j) =\nicefrac{\alpha_j}{d_j}$ and $\sin(\gamma_j) =\nicefrac{\beta_j}{d_j}$. Then $v_t$ can be rewritten as $v_t=\sum\limits_{j=1}^{m}d_j\cos(2\pi \theta_j t-\gamma_j)$. 

Let $\gamma=-(\frac{\gamma_1}{2\pi}, \ldots, \frac{\gamma_m}{2\pi})$ be in the $m$-dimensional torus 
$(\mathbb{R}/\mathbb{Z})^m$. Then the set $\{(\gamma+(\theta_1 t, \ldots, \theta_m t))\in (\mathbb{R}/\mathbb{Z})^m, t \in \mathbb{N}\}$ is dense in $(\mathbb{R}/\mathbb{Z})^m$, for more details see Corollary 4.2.3 of \cite{MK}.
Consequently, $\liminf\limits_{t\rightarrow +\infty}|v_t|=0$.
\end{proof}

\section{Examples}\label{examples}
In this section, we consider  the celebrated Brownian oscillator 
\begin{equation}\label{unifm}
\ddot{x}_t+\gamma \dot{x}_t+\kappa x_t=\epsilon \dot{B}_t
\quad \textrm{ for any } t\geq 0,
\end{equation}
where  
$x_t$ denotes the position at time $t$  of the holding mass $m$ with respect to its equilibrium position, 
$\gamma>0$ denotes the damping constant,
$\kappa>0$ denotes the restoration constant (Hooke's constant) and $(B_t:t\geq 0)$ is a Brownian motion.
{For each initial displacement from the equilibrium position $x_0=u$ and  initial velocity $\dot{x}_0=v$ we have a unique solution of \eqref{unifm}.} For further details see Chapter $8$ in \cite{Mao}.

Without loss of generality we can assume that the mass $m$ is one. 
Using the classical {\it{forward difference approximation}} with the step size $h>0$ (fixed), we obtain 
\[
\frac{1}{h^2} (x_{(n+2)h}-2x_{(n+1)h}+x_{nh})+\frac{\gamma}{h}(x_{(n+1)h}-x_{nh}) +\kappa x_{nh} = \frac{\epsilon}{h}(B_{(n+3)h}-B_{(n+2)h})
\]
for any $n\in \mathbb{N}_0$ with the initial data $x_0=u$ and $x_h=x_0+ \dot{x}_0h= u+vh$.
For consistency, let $X_t=x_{th}$ for any $t\in \mathbb{N}_0$. The latter can be rewritten as
\begin{equation}\label{nusc}
X_{t+2}=\left( 2-\gamma h\right) X_{t+1} - \left( 1-\gamma h +\kappa h^2\right) X_{t} +  \epsilon h(B_{(t+3)h}-B_{(t+2)h}) \quad \textrm{for any } t\in \mathbb{N}_0.
\end{equation}
Notice that the sequence $(B_{(t+3)h}-B_{(t+2)h}:t\in \mathbb{N}_0)$ are i.i.d. random variables with Gaussian distribution with zero mean and variance $h$. 
Therefore
\begin{equation*}
X_{t+2}=\left( 2-\gamma h\right) X_{t+1} - \left( 1-\gamma h +\kappa h^2\right) X_{t} +  \epsilon h^{\nicefrac{3}{2}}\xi_{t+2}\quad \textrm{for any } t\in \mathbb{N}_0,
\end{equation*}
where $(\xi_{t+2}:t\in \mathbb{N}_0)$ is a sequence of i.i.d. random variables with  standard Gaussian distribution. 
This is exactly a linear recurrence of degree $2$ with control sequence $(\epsilon h^{\nicefrac{3}{2}} \xi_{t+2} : t\in \mathbb{N}_0)$, and its characteristic polynomial is given by
\begin{equation}\label{exchara}
\lambda^2+(\gamma h -2)\lambda+(1-\gamma h + \kappa h^2).
\end{equation}
To fulfill assumption \eqref{H} we deduce the following conditions.
\begin{itemize}
\item[i)] If $\gamma^2-4k>0$, then polynomial \eqref{exchara} has two distinct real roots. In this case a  sufficient condition to verify  \eqref{H} is $h\in (0,\nicefrac{2}{\gamma})$.
\item[ii)] If $\gamma^2-4k=0$, then polynomial \eqref{exchara} has two repeated real roots.
In this case \eqref{H} is equivalent to $h\in (0,\nicefrac{\gamma}{\kappa})$.
\item[iii)] If $\gamma^2-4k<0$, then polynomial \eqref{exchara} has two complex conjugate roots. In this case \eqref{H} is equivalent to
$h\in (0,\nicefrac{\gamma}{\kappa})$.
\end{itemize}
In other words, there exists $h^*\in (0,1)$ such that for each $h\in (0,h^*)$ the characteristic polynomial \eqref{exchara} satisfies assumption \eqref{H}.
From here to the end of this section, we assume that $h\in (0,h^*)$.

Now, we compute  $r$, $l$, $v_t$ and  $\mathcal{C}$ which appear in Lemma \ref{lya}.
Let $\lambda_1$ and $\lambda_2$  be roots of \eqref{exchara}.
Denote $r_1=\|\lambda_1\|$ and $r_2=\|\lambda_2\|$. Recall the function $\sg(\cdot)$ is 
defined over the domain $\mathbb{R}\setminus\{0\}$ by
$\sg(x)=\nicefrac{x}{|x|}$.
We assume that $(x_0,x_1)\not= (0,0)$.
We analyze as far as possible when the conditions of Theorem \ref{main} are fulfilled for the model \eqref{nusc}.
\begin{itemize}
\item[i)] 
{\bf{Real roots with different absolute values.}}
$\lambda_1$ and $\lambda_2$ are real and $r_1 \not= r_2$. In this case, 
\[x_t = c_1 \lambda_1^t+c_2 \lambda_2^t\quad \textrm{ for any } t\in \mathbb{N}_0,\]
where $c_1$ and $c_2$ are unique real constants given by initial data $x_0,x_1$. 
Since $(x_0,x_1)\not= (0,0)$ then $(c_1,c_2)\not=(0,0)$.
 Without loss of generality assume that $r_1>r_2$. 
\begin{itemize}
\item[i.1)] If $c_1\neq 0$ then
\[
\lim\limits_{t\rightarrow +\infty}\left|\frac{x_t}{r_1^t}-c_1(\sg(\lambda_1))^t\right|=0.
\]
\item[i.2)] If $c_1=0$ then $c_2\neq 0$. Therefore
\[
\lim\limits_{t\rightarrow +\infty}\left|\frac{x_t}{r_2^t}-c_2(\sg(\lambda_2))^t\right|=0.
\]
\end{itemize}
Consequently, $\mathcal{C}=\mathbb{R}^2\setminus\{(0,0)\}$.
\item[ii)]
{\bf{Real roots with the same absolute value.}}
$\lambda_1$ and $\lambda_2$ are real and $r \coloneqq     r_1 = r_2$.
\begin{itemize}
\item[ii.1)] If $\lambda_1=\lambda_2=r\sg (\lambda_1)$ then 
\[
x_t=c_1 r^t(\sg (\lambda_1))^t+c_2 t r^t(\sg (\lambda_1))^t \quad\textrm{ for any } t\in \mathbb{N}_0,\]
where $c_1$ and $c_2$ are unique real constants given by initial data $x_0,x_1$. 
Since $(x_0,x_1)\not= (0,0)$ then $(c_1,c_2)\not=(0,0)$.
Then
\begin{itemize}
\item[ii.1.1)] If $c_2\neq 0$ then
	\begin{equation*}
	\lim\limits_{t\rightarrow +\infty}\left|\frac{x_t}{tr^t}-c_2(\sg(\lambda_1))^t\right|=0.
	\end{equation*}
\item[ii.1.2)] If  $c_2 = 0$ then $c_1\neq 0$. Therefore
	\begin{equation*}
	\lim\limits_{t\rightarrow +\infty}\left|\frac{x_t}{r^t}-c_1(\sg(\lambda_1))^t\right|=0.
	\end{equation*}
\end{itemize}
Consequently, $\mathcal{C}=\mathbb{R}^2\setminus\{(0,0)\}$.
\item[ii.2)] If $\lambda_1\neq\lambda_2$
then
\[
x_t=c_1 r^t+c_2 (-r)^t \quad\textrm{ for any } t\in \mathbb{N}_0,\]
where $c_1$ and $c_2$ are unique real constants given by initial data $x_0,x_1$. Therefore
\begin{equation*}
\lim\limits_{t\rightarrow +\infty}\left|\frac{x_t}{r^t}- (c_1+c_2(-1)^t)\right|=0.
\end{equation*}
Consequently, 
\begin{align*}
\mathcal{C}=&\{(x_0,x_1)\in \mathbb{R}^2:c_1+c_2\neq0 \textrm{ and } c_1-c_2\neq 0\}\\
=&\{(x_0,x_1)\in \mathbb{R}^2:x_0\neq 0 \textrm{ and } x_1 \neq 0\}.
\end{align*}
\end{itemize}
\item[iii)] {\bf{Complex conjugate roots.}}
Since the coefficients of the characteristic polynomial are real if $\lambda$ is a root of the polynomial, then conjugate $\overline{\lambda}$ is also a root. We can assume that
$\lambda_1=re^{i2\pi\theta}$ and $\lambda_2=re^{-i2\pi\theta}$ with $r\in (0,1)$ and 
$\theta \in (0,1)\setminus \{\nicefrac{1}{2}\}$. 
In this setting
\[x_t=c_1r^t \cos(2\pi\theta t)+c_2r^t \sin(2\pi\theta t) \quad\textrm{ for any } t\in \mathbb{N}_0, \] 
where $c_1$ and $c_2$ are unique real constants given by initial data $x_0,x_1$. 
Thus
\[
\lim\limits_{t\rightarrow +\infty}\left|\frac{x_t}{r^t}-(c_1 \cos(2\pi\theta t)+c_2 \sin(2\pi\theta t))\right|=0.
\]
Since $(x_0,x_1)\not= (0,0)$ then $(c_1,c_2)\not=(0,0)$. 
Let $c=\sqrt{c_1^2+c_2^2}$,
$\cos(\gamma) =\nicefrac{c_1}{c}$ and
$\sin(\gamma) =\nicefrac{c_2}{c}$. Consequently,
\[v_t:=
c_1 \cos(2\pi\theta t)+c_2 \sin(2\pi\theta t)
=c\cos(2\pi\theta t -\gamma) \quad\textrm{ for any } t\in \mathbb{N}_0.\]
Observe that $\gamma$ depends on the initial data $x_0$ and $x_1$.
Let us analyze under which conditions on $x_0$ and $x_1$ we have $\liminf\limits_{t\rightarrow+\infty}|v_t|>0$. 
\begin{itemize}
\item[iii.1)] If $\theta$ is a rational number 
 then the sequence $(\cos(2\pi\theta t-\gamma),t\in \mathbb{N}_0)$  takes finite number of values. Notice that there exists $t_0\in \mathbb{N}_0$ such that $2\pi\theta t_0-\gamma = \nicefrac{\pi}{2}+k\pi$ for some $k\in \mathbb{Z}$, if and only if $\cos(2\pi\theta t_0-\gamma)=0$.
Therefore, $\liminf\limits_{t\rightarrow+\infty}|v_t|>0$ if and only if
\begin{equation*}
\mathcal{C}=\{(x_0,x_1)\in \mathbb{R}^2: 2\pi\theta t-\gamma\not=\frac{\pi}{2} + k\pi \quad\textrm{ for any } t\in\mathbb{N}_0,~ k\in\mathbb{Z}\}.
\end{equation*}
\item[iii.2)]If $\theta$ is an irrational number. Then by Corollary 4.2.3 of \cite{MK} the set  $\{(\theta t-\nicefrac{\gamma}{2\pi})\in\mathbb{R}/\mathbb{Z}:t\in \mathbb{N}_0\}$ is dense in the circle $\mathbb{R}/\mathbb{Z}$ and consequently the set
$\{\cos(2\pi\theta t-\gamma): t\in \mathbb{N}_0\}$ is dense in $[-1,1]$. 
Therefore, for any $\gamma$ we have $\liminf\limits_{t\rightarrow+\infty}|v_t|=0$, which implies $\mathcal{C}=\emptyset$.
\end{itemize}
\end{itemize}

\appendix
\section{Variance Representation of $X_t^{(\epsilon)}$}\label{tools}
Since  $(\xi_t:t\geq 0)$ is a sequence of i.i.d. random variables with standard Gaussian distribution, it is not hard to see that for any $t\geq p$ the random variable 
$X_t^{(\epsilon)}$ has Gaussian distribution, whose expectation is $x_t$. 
The next lemma provides a representation of its variance  under  assumption 
\eqref{H}. 

{Now, for the sake of intuitive reasoning and in a conscious abuse of notation we introduce the following notation.}
For each $s\in \mathbb{N}_0$ denote by 
$\sum k_j=s$ the set 
\[\left\{(k_1,\dots,k_p)\in \mathbb{N}^p_0:~ \sum\limits_{j=1}^p k_j=s\right\}\]
and denote by 
$\sum\limits_{\sum k_j=s}$ the sum of 
$\sum\limits_{(k_1,\ldots,k_p)\in \sum k_j=s}$. 
\begin{lemma}\label{distri}
Assume that \eqref{H} holds.
For any $t\geq p$, $X^{(\epsilon)}_t$ has Gaussian distribution with mean $x_t$ and variance $\epsilon^2\sigma^2_t$, where  
\[
\sigma^2_t=1+
\left(\sum\limits_{\sum k_j=1}\lambda_1^{k_1}\cdots\lambda_p^{k_p}\right)^2
+\cdots+
\left(\sum\limits_{\sum k_j=t-p}\lambda_1^{k_1}\cdots\lambda_p^{k_p}\right)^2\]
and $\lambda_1,\ldots,\lambda_p$ are the roots of 
\eqref{ce}.
\end{lemma}

\begin{proof}
By the superposition principle, the solution of the non-homogeneous linear recurrence \eqref{ar} can be written as the general solution of the homogeneous linear recurrence \eqref{plr} plus a particular solution of the non-homogeneous linear recurrence \eqref{ar} as follows:
\[
X^{(\epsilon)}_t= x^{\textrm{gen}}_t+X^{(\textrm{par},\epsilon)}_t \quad\textrm{ for any } t\in \mathbb{N}_0,
\]
where $X^{(\textrm{par},\epsilon)}_t$ solves the non-homogeneous linear recurrence \eqref{ar}, $x^{\textrm{gen}}_t$ solves the homogeneous linear recurrence \eqref{plr} but possible both solutions do not fit the prescribed initial conditions. The initial conditions are fitting after adding themselves. For more details see Section $2.4$ of \cite{SA}. 

To find a  particular solution, we introduce the {\it{Lag operator}} $\mathbb{L}$ 
which acts as follows: $x_{t-1}=\mathbb{L} \circ x_{t}$. The inverse operator $\mathbb{L}^{-1}$ is defined as $\mathbb{L}^{-1}\circ x_t=x_{t+1}$.
For more details about the Lag operator we recommend Chapter $2$ of \cite{JDH}. 
Notice that the random linear recurrence \eqref{ar} can be rewritten as
\[
(\mathbb{L}^{-p}-\phi_1 \mathbb{L}^{-p+1}- \cdots - \phi_p)\circ X^{(\textrm{par},\epsilon)}_t=\epsilon\mathbb{L}^{-p} \circ\xi_t.
\]
Then
\[
(1-\lambda_1\mathbb{L})(1-\lambda_2\mathbb{L})\cdots(1-\lambda_p\mathbb{L})\circ X^{(\textrm{par},\epsilon)}_t=\epsilon\xi_t,
\]
where $\lambda_1,\ldots,\lambda_p$ are the roots of \eqref{ce}. Since the modules of the roots of \eqref{ce} are strictly less than one then
\[
X^{(\textrm{par},\epsilon)}_t=(1+\lambda_1\mathbb{L}+\lambda_1^2\mathbb{L}^2+\cdots)\cdots
(1+\lambda_p\mathbb{L}+\lambda_p^2\mathbb{L}^2+\cdots) \circ \epsilon \xi_t
\]
for any $t\geq p$. Since $\xi_t$ is only defined for  $t\geq p$, then
\[
X^{(\textrm{par},\epsilon)}_t=\left(1+\sum\limits_{\sum k_i=1}\lambda_1^{k_1}\cdots\lambda_p^{k_p}\mathbb{L}+\dots+\sum\limits_{\sum k_i=t-p}\lambda_1^{k_1}\cdots\lambda_p^{k_p}\mathbb{L}^{t-p}\right)\circ \epsilon \xi_t.
\]
Consequently,
\begin{equation}\label{casa}
X^{(\epsilon)}_t=x^{\textrm{gen}}_t+
\epsilon\left(\xi_{t}+
\sum\limits_{\sum k_i=1}\lambda_1^{k_1}\cdots\lambda_p^{k_p}\xi_{t-1}
+\cdots+\sum\limits_{\sum k_i=t-p}\lambda_1^{k_1}\cdots\lambda_p^{k_p}\xi_{p}\right)
\end{equation}
for $t\geq p$, where $x^{\textrm{gen}}_t$ satisfies \eqref{plr}.
After fitting the initial conditions, we see that $(x^{\textrm{gen}}_t:t\in \mathbb{N}_0)$ is the solution of \eqref{plr} with initial data $x_0,\ldots,x_{p-1}$.
Therefore $x^{\textrm{gen}}_t=x_t$ for any $t\in \mathbb{N}_0$.
Since $(\xi_t:t\geq p)$ are i.i.d.
Gaussian random variables with zero mean and unit variance then  for $t\geq p$,
$X^{(\epsilon)}_t$ is a Gaussian distribution. Therefore it is characterized by its mean and variance. Since the expectation of $X^{(\epsilon)}_t$ is $x_t$ then we only need to compute its variance. From \eqref{casa} we get
\[
\textrm{Var}\left(X^{(\epsilon)}_t\right)=\epsilon^2\left(1+
\left(\sum\limits_{\sum k_j=1}\lambda_1^{k_1}\cdots\lambda_p^{k_p}\right)^2
+\cdots+
\left(\sum\limits_{\sum k_j=t-p}\lambda_1^{k_1}\cdots\lambda_p^{k_p}\right)^2\right)\]
for any $t\geq p$.
\end{proof}

\begin{lemma}\label{ergodico}
Assume that \eqref{H} holds.
As $t$ goes to infinity, $X^{(\epsilon)}_t$ converges in the total variation distance to a random variable
$X^{(\epsilon)}_\infty$ that has Gaussian distribution with zero mean and variance $\epsilon^2\sigma^2_\infty\in [\epsilon^2,+\infty)$.
\end{lemma}

\begin{proof}
From Lemma \ref{distri} we have that for any $t\geq p$,
$X^{(\epsilon)}_t$ has mean $x_t$ which is the solution of \eqref{plr} and variance $\epsilon^2\sigma^2_t$ where
\[
\sigma^2_t=1+
\left(\sum\limits_{\sum k_j=1}\lambda_1^{k_1}\cdots\lambda_p^{k_p}\right)^2
+\cdots+
\left(\sum\limits_{\sum k_j=t-p}\lambda_1^{k_1}\cdots\lambda_p^{k_p}\right)^2.
\]
Since all the roots of  \eqref{ce} have modulus strictly less than one, with \eqref{eq1} $x_t$ converges to zero when $t$ goes to infinity.  
By a counting argument we can see that for any $s\in \mathbb{N}_0$
\[\textrm{Card}\left(\sum k_j = s\right)\leq (s+1)^p,\]
where Card denotes the cardinality of the given set. Then for any $t\geq p$
\begin{align*}
\sigma^2_t&=1+
\left(\sum\limits_{\sum k_j=1}\lambda_1^{k_1}\dots\lambda_p^{k_p}\right)^2
+\cdots+
\left(\sum\limits_{\sum k_j=t-p}\lambda_1^{k_1}\dots\lambda_p^{k_p}\right)^2\\
&\leq 1+(2^p \kappa)^2+\cdots+((t-p+1)^p \kappa^{t-p})^2\\
&=
\sum\limits_{j=0}^{t-p}(j+1)^{2p}{\kappa}^{2j}\leq \sum\limits_{j=0}^{\infty}(j+1)^{2p}{\kappa}^{2j}<+\infty,
\end{align*}
where $\kappa=\max\limits_{1\leq j\leq n}{|\lambda_j|}<1$.
Since $1\leq \sigma^2_{t}\leq  \sigma^2_{t+1}\leq \sum\limits_{j=0}^{\infty}(j+1)^{2p}{\kappa}^{2j}<+\infty$ for any $t\geq p$  then  
$\lim\limits_{t\rightarrow +\infty}{\sigma^2_t}$ exists. Denote by 
$\sigma^2_\infty$ its value, then  $\sigma^2_\infty\in [1,+\infty)$. 
It follows from 
Lemma \ref{pend1} that $X^{(\epsilon)}_t$ converges in the total variation distance to $X^{(\epsilon)}_\infty$ as $t$ goes to infinity, which 
has Gaussian distribution with zero mean and variance $\epsilon^2\sigma^2_\infty$.
\end{proof} 

\section{Total Variation Distance between Gaussian distributions}\label{totalvariation}
In this section we provide some useful properties for the total variation distance between Gaussian distributions.
Recall that $\N(m,\sigma^2)$ denotes the Gaussian distribution with mean $m\in \mathbb{R}$ and variance $\sigma^2\in (0,+\infty)$.
A straightforward computation leads
\begin{equation}\label{alter}
\tv\left(\N\left(m_1,\sigma^2_1\right),\N\left(m_2,\sigma^2_2\right)\right)=\frac{1}{2}\int\limits_{\mathbb{R}}\left|\frac{1}{\sqrt{2\pi}\sigma_1}e^{-\frac{(x-m_1)^2}{2\sigma_1^2}}-\frac{1}{\sqrt{2\pi}\sigma_2}e^{-\frac{(x-m_2)^2}{2\sigma_2^2}}\right|\ud x
\end{equation}
for any $m_1,m_2\in \mathbb{R}$, $\sigma^2_1,\sigma^2_2\in (0,+\infty)$.
For details see Lemma $3.3.1$ in \cite{RDR}.
\begin{lemma}\label{pend}
Let $m_1,m_2\in \mathbb{R}$ and $\sigma^2_1, \sigma^2_2 \in (0,+\infty)$. Then
\begin{itemize}
\item[i)] $\tv(\N(m_1,\sigma^2_1),\N(m_2,\sigma^2_2))=
\tv (\N(m_1-m_2,\sigma^2_1),\N(0,\sigma^2_2)).$
\item[ii)]
$\tv(\N(cm_1,c^2\sigma^2_1),\N(cm_2,c^2\sigma^2_2))=
\tv(\N(m_1,\sigma^2_1),\N(m_2,\sigma^2_2))$ 
for any $c\not=0$.
\end{itemize}
\end{lemma}
\begin{proof}
The proofs of item i) and item ii) proceed from the Change of Variable Theorem.
\end{proof}

\begin{lemma}\label{pend2}~
\begin{itemize}
\item[i)] For any $m \in \mathbb{R}$ and $\sigma^2\in (0,+\infty)$ we have
\[\tv (\N(m,\sigma^2),\N(0,\sigma^2))=
\frac{2}{\sqrt{2\pi}}\int\limits_{0}^{\frac{|m|}{2\sigma}}e^{-\frac{x^2}{2}}\ud x\leq \frac{|m|}{\sigma\sqrt{2\pi}}.\]
\item[ii)] For any
$m_1,m_2\in \mathbb{R}$  and $\sigma^2\in (0,+\infty)$ such that 
 $|m_1|\leq |m_2|<+\infty$ 
 we have 
\[\tv (\N(m_1,\sigma^2),\N(0,\sigma^2))\leq
\tv (\N(m_2,\sigma^2),\N(0,\sigma^2)).
\]
\item[iii)] 
If $\lim\limits_{t\rightarrow +\infty}|m_t|=+\infty$ and $\sigma^2\in (0,+\infty)$ then 
\[
\lim\limits_{t\rightarrow +\infty}\tv (\N(m_t,\sigma^2),\N(0,\sigma^2))=1.
\]
\end{itemize}
\end{lemma}

\begin{proof}
Notice that item ii) and item iii) follow immediately from item i). Therefore we only prove item i).
From item ii) of Lemma \ref{pend} we can assume that $m\geq 0 $.
Observe that
\[
\begin{split}
\tv (\N(m,\sigma^2),\N(0,\sigma^2))&=
\frac{1}{2\sqrt{2\pi}\sigma}
\int\limits_{-\infty}^{\frac{m}{2}}
\left(e^{-\frac{x^2}{2\sigma^2}}-e^{-\frac{(x-m)^2}{2\sigma^2}}\right)\ud x\\
&\hspace{0.3cm}+\frac{1}{2\sqrt{2\pi}\sigma}
\int\limits_{\frac{m}{2}}^{+\infty}
\left(e^{-\frac{(x-m)^2}{2\sigma^2}}
-e^{-\frac{x^2}{2\sigma^2}}\right)\ud x\\
&=
\frac{2}{\sqrt{2\pi}\sigma}\int\limits_{0}^{\frac{m}{2}}
e^{-\frac{x^2}{2\sigma^2}}\ud x.
\end{split}
\]
The latter easily implies the result.

\end{proof}

\begin{lemma}\label{formulita}
For any $\sigma^2\in (0,1)\cup (1,+\infty)$ we have
\[\tv (\N(0,\sigma^2),\N(0,1))
= \frac{2}{\sqrt{2\pi}}\int\limits_
{\min\left\{x(\sigma),\frac{x(\sigma)}{\sigma}\right\}}^{\max\left\{x(\sigma),\frac{x(\sigma)}{\sigma}\right\}}
e^{-\frac{x^2}{2}}\ud x\leq \frac{2}{\sqrt{2\pi}}x(\sigma)\left|\nicefrac{1}{\sigma}-1\right|,
\]
where 
$x(\sigma)=
\sigma\left(
\frac{\ln({\sigma^2})}{\sigma^2-1}
\right)^{\nicefrac{1}{2}}$.
Moreover, we have $\lim\limits_{\sigma^2\to 1}x(\sigma)=1$.
\end{lemma}
\begin{proof}
In this case a formula for $\tv (\N(0,\sigma^2),\N(0,1))$ can be computed explicitly as we did in the proof of item i) of Lemma \ref{pend2}. 
Indeed, if 
$\sigma^2\in (0,1)$ observe that
\[
\begin{split}
\tv (\N(0,\sigma^2),\N(0,1))
&=  \frac{1}{2\sqrt{2\pi}}\int\limits_
{-\infty}^{+\infty}
\left|\frac{1}{\sigma} e^{-\frac{x^2}{2\sigma^2}}  - e^{-\frac{x^2}{2}} \right| \ud x  =   \frac{1}{\sqrt{2\pi}}\int\limits_
{0}^{+\infty}
\left|\frac{1}{\sigma} e^{-\frac{x^2}{2\sigma^2}}  - e^{-\frac{x^2}{2}} \right| \ud x 
\\
&\hspace{-2.5cm}= \frac{1}{\sqrt{2\pi}} \left[ \int\limits_
{0}^{x(\sigma)}  \left( \frac{1}{\sigma} e^{-\frac{x^2}{2\sigma^2}}  - e^{-\frac{x^2}{2}} \right) \ud x + \int\limits_
{x(\sigma)}^{+\infty} \left( e^{-\frac{x^2}{2}}-\frac{1}{\sigma} e^{-\frac{x^2}{2\sigma^2}}  \right) \ud x \right]\\
&\hspace{-2.5cm}= \frac{2}{\sqrt{2\pi}}\int\limits_
{0}^{x(\sigma)}
\left(\frac{1}{\sigma} e^{-\frac{x^2}{2\sigma^2}}  - e^{-\frac{x^2}{2}} \right) \ud x
= \frac{2}{\sqrt{2\pi}}\int\limits_
{x(\sigma)}^{\frac{x(\sigma)}{\sigma}}
e^{-\frac{x^2}{2}}\ud x\leq 
\frac{2}{\sqrt{2\pi}}x(\sigma)(\nicefrac{1}{\sigma}-1).
\end{split}
\]

On the other hand, if  $\sigma^2\in (1,+\infty)$ one can also deduce that
\[\tv (\N(0,\sigma^2),\N(0,1))
= \frac{2}{\sqrt{2\pi}}\int\limits_
{\frac{x(\sigma)}{\sigma}}^{{x(\sigma)}}
e^{-\frac{x^2}{2}}\ud x\leq 
\frac{2}{\sqrt{2\pi}}x(\sigma)(1-\nicefrac{1}{\sigma}).
\]
The second part of the lemma is a direct computation.
\end{proof}

\begin{lemma}[Continuity]\label{pend1}~
If $\lim\limits_{t\rightarrow +\infty}m_t=m \in \mathbb{R}$ 
and  $\lim\limits_{t\rightarrow +\infty}\sigma^2_t=\sigma^2 \in (0,+\infty)$ 
then
\[
\lim\limits_{t\rightarrow +\infty}\tv (\N(m_t,\sigma^2_t),\N(m,\sigma^2))=0.
\]
\end{lemma}
\begin{proof}
The proof follows from the triangle inequality together with item i) of Lemma \ref{pend}, item i) of Lemma \ref{pend2} and Lemma \ref{formulita}.
\end{proof}

\begin{lemma}\label{pend4}~
Let $\sigma^2\in (0,+\infty)$.
\begin{itemize}
\item[i)]
If $\limsup\limits_{t\rightarrow +\infty}|m_t|\leq C_0\in [0,+\infty)$ then 
\[
\limsup\limits_{t\rightarrow +\infty}\tv(\N(m_t,\sigma^2),\N(0,\sigma^2))\leq \tv(\N(C_0,\sigma^2),\N(0,\sigma^2)).
\]
\item[ii)] If $\liminf\limits_{t\rightarrow +\infty}|m_t|\geq  C_1\in [0,+\infty)$ then 
\[
\liminf\limits_{t\rightarrow +\infty}\tv(\N(m_t,\sigma^2),\N(0,\sigma^2))\geq  \tv(\N(C_1,\sigma^2),\N(0,\sigma^2)).
\]
\end{itemize}
\begin{proof}~
\begin{itemize}
\item[i)] Let 
$L:=\limsup\limits_{t\rightarrow +\infty}\tv(\N(m_t,\sigma^2),\N(0,\sigma^2))$. Then there exists a subsequence $(t_n:n\in \mathbb{N})$ such that
$\lim\limits_{n\rightarrow +\infty}t_n=+\infty$ and
\[\lim\limits_{n\rightarrow +\infty}\tv(\N(m_{t_n},\sigma^2),\N(0,\sigma^2))=L.\]
Since 
$\limsup\limits_{t\rightarrow +\infty}|m_t|\leq C_0$ then 
$\limsup\limits_{n\rightarrow +\infty}|m_{t_n}|\leq C_0$. Then again
there exists a subsequence 
$(t_{n_k}:k\in \mathbb{N})$ of
$(t_n:n\in \mathbb{N})$ such that
$\lim\limits_{k\rightarrow +\infty}t_{n_k}=+\infty$ and
$\lim\limits_{k\rightarrow +\infty}
|m_{t_{n_k}}|$ exists. Let $C:=\lim\limits_{k\rightarrow +\infty}
|m_{t_{n_k}}|$ and notice that $0\leq C\leq C_0$. From Lemma \ref{pend1} we obtain
\[\lim\limits_{k\rightarrow +\infty}\tv(\N(m_{t_{n_k}},\sigma^2),\N(0,\sigma^2))=
\tv(\N(C,\sigma^2),\N(0,\sigma^2)).\]
Notice that $\lim\limits_{k\rightarrow +\infty}\tv(\N(m_{t_{n_k}},\sigma^2),\N(0,\sigma^2))=L$, then
by item ii) of Lemma \ref{pend2}
we deduce
\[
L=\tv(\N(C,\sigma^2),\N(0,\sigma^2))\leq 
\tv(\N(C_0,\sigma^2),\N(0,\sigma^2)).\]
\item[ii)] The proof of item ii) follows from similar arguments as we did in item i). We left the details to the interested reader.
\end{itemize}
\end{proof}
\end{lemma}

\section{Tools}\label{toolsappendix}
In this section we state some elementary tools that we used along the article. 
We state here for the sake of completeness. 

\begin{lemma} \label{C2}
Let $(a_\epsilon:\epsilon>0)$ and $(b_\epsilon:\epsilon>0)$ be functions of real numbers.
Assume that 
$\lim\limits_{\epsilon\to 0^+}{b_\epsilon}=b\in \mathbb{R}$. Then 
\begin{itemize}
\item[i)] $\limsup\limits_{\epsilon\to 0^+}(a_\epsilon+b_\epsilon)=\limsup\limits_{\epsilon\to 0^+}a_\epsilon+b$.
\item[ii)] $\liminf\limits_{\epsilon\to 0^+}(a_\epsilon+b_\epsilon)=\liminf\limits_{\epsilon\to 0^+}a_\epsilon+b$.
\item[iii)] $\liminf\limits_{\epsilon\to 0^+}(a_\epsilon b_\epsilon)=b\liminf\limits_{\epsilon\to 0^+}a_\epsilon$~~ when $b>0$.
\end{itemize}
\end{lemma}
\begin{proof}
The proofs proceed by definition of limit superior and limit inferior using subsequences. 
\end{proof}

\begin{lemma} \label{C3}
For any $\alpha\in\mathbb{R}$ and $r \in (0,1)$ we have
\[
\lim\limits_{\epsilon\rightarrow0^+}\frac{(t^{(\epsilon)})^{\alpha} r^{t^{(\epsilon)}}}{\epsilon} = 1,
\]
where
$
t^{(\epsilon)}=\frac{\ln(\nicefrac{1}{\epsilon})}{\ln(\nicefrac{1}{r})}+
\alpha\frac{\ln\left(\frac{\ln(\nicefrac{1}{\epsilon})}{\ln(\nicefrac{1}{r})}\right)}{\ln(\nicefrac{1}{r})}.
$
\end{lemma}
\begin{proof}
Notice that
$
t^{(\epsilon)}=\log_{r}(\epsilon)-\alpha\log_r(\log_r (\epsilon)).
$
A straightforward computation shows
\[
\lim\limits_{\epsilon\rightarrow0^+}\frac{(t^{(\epsilon)})^{\alpha} r^{t^{(\epsilon)}}}{\epsilon} = \lim\limits_{\epsilon\rightarrow0^+}\left( 1-\alpha\frac{\log_r(\log_r(\epsilon))}{\log_r(\epsilon)}\right)^{\alpha}=1.
\]
\end{proof}

\section*{Acknowledgments}
G. Barrera 
gratefully acknowledges support from a post-doctorate 
Pacific Institute for the Mathematical Sciences (PIMS, 2017--2019) grant held
at the Department of Mathematical and Statistical Sciences at University of Alberta. Both authors would like to express their gratitude to University of Alberta for all the facilities used  along the realization of this work.

\bibliographystyle{amsplain}

\end{document}